\documentclass[11pt]{amsart}
\usepackage[T1]{fontenc}
\usepackage[margin=1in]{geometry}
\usepackage[foot]{amsaddr}
\usepackage{amsmath,amssymb,amsthm}
\usepackage{graphicx} 
\usepackage[pdfusetitle]{hyperref}
\usepackage[capitalise]{cleveref}
\usepackage{mathabx}
\usepackage[dvipsnames]{xcolor}
\usepackage[style=trad-abbrv]{biblatex}
\usepackage{xcolor}
\newtheorem{theorem}{Theorem}
\newtheorem{lemma}[theorem]{Lemma}
\newtheorem{conjecture}[theorem]{Conjecture}

\DeclareMathOperator{\pw}{pw}
\DeclareMathOperator{\tw}{tw}
\DeclareMathOperator{\bw}{bw}

\DeclareMathOperator{\circumference}{circumference}
\DeclareMathOperator{\cocircumference}{cocircumference}

\title{Pathwidth vs cocircumference}

\author[
  M.~Briański \and
  G.~Joret \and
  M.~T.~Seweryn
]{
  Marcin Briański \and
  Gwenaël Joret \and
  Michał T. Seweryn
}

\address[M.~Briański]{
  Theoretical Computer Science Department\\ 
  Faculty of Mathematics and Computer Science\\ 
  Jagiellonian University\\ 
  Kraków, Poland
}
\address[G.~Joret, M.~T.~Seweryn]{
  Computer Science Department\\
  Université libre de Bruxelles\\
  Brussels\\
  Belgium
}

\email{marcin.brianski@doctoral.uj.edu.pl}
\email{gwenael.joret@ulb.be}
\email{michal.seweryn@ulb.be}

\thanks{M.\ Briański is partially supported by the Polish National Science Centre grant 2019/34/E/ST6/00443. G.\ Joret and M.T.\ Seweryn are supported by a PDR grant from the Belgian National Fund for Scientific Research (FNRS)}

\date{\today}

\begin{document}

\begin{abstract}
The {\em circumference} of a graph $G$ with at least one cycle is the length of a longest cycle in $G$.  
A classic result of Birmelé (2003) states that the treewidth of $G$ is at most its circumference minus $1$. 
In case $G$ is $2$-connected, this upper bound also holds for the pathwidth of $G$; in fact, even the treedepth of $G$ is upper bounded by its circumference (Briański, Joret, Majewski, Micek, Seweryn, Sharma; 2023). 
In this paper, we study whether similar bounds hold when replacing the circumference of $G$ by its {\em cocircumference}, defined as the largest size of a {\em bond} in $G$, an inclusion-wise minimal set of edges $F$ such that $G-F$ has more components than $G$. 
In matroidal terms, the cocircumference of $G$ is the circumference of the bond matroid of $G$. 

Our first result is the following `dual' version of Birmelé's theorem: The treewidth of a graph $G$ is at most its cocircumference.
Our second and main result is an upper bound of $3k-2$ on the pathwidth of a $2$-connected graph $G$ with cocircumference $k$. 
Contrary to circumference, no such bound holds for the treedepth of $G$.  
Our two upper bounds are best possible up to a constant factor. 
\end{abstract}

\maketitle

\section{Introduction}

The {\em circumference} of a graph $G$ with at least one cycle is the length of a longest cycle in $G$.  
The following is a well-known upper bound on the treewidth of graphs. 

\begin{theorem}[Birmelé~\cite{B03}]
\label{thm:tw_circ}
Every graph with at least one cycle has treewidth at most its circumference minus $1$. 
\end{theorem}

For $2$-connected graphs, it was recently strengthened as follows. 

\begin{theorem}[Briański, Joret, Majewski, Micek, Seweryn, and Sharma~\cite{TreedepthVsCircumference}]
\label{thm:td_circ}
Let $G$ be a $2$-connected graph with treedepth $d$ and circumference $k$. 
Then every edge of $G$ is included in a cycle of length at least $d$. 
Thus, $G$ has treedepth at most $k$, and in particular pathwidth at most $k-1$. 
\end{theorem}

(Note that $2$-connectivity is essential in \cref{thm:td_circ}, even if we only wish to bound the pathwidth: 
A complete ternary tree of height $h$ has pathwidth $h$, so the union of such a tree with a triangle 
has pathwidth $h$ but circumference $3$.) 

In this paper, we study whether similar upper bounds hold when replacing the circumference of $G$ by its {\em cocircumference}, defined as the largest size of a {\em bond} in $G$, an inclusion-wise minimal set of edges $F$ such that $G-F$ has more components than $G$. 
We remark that the cocircumference of $G$ is the cocircumference of the cycle matroid of $G$, or equivalently, the circumference of its dual matroid, the bond matroid of $G$. 
Thus, in this sense we are looking for `dual' versions of Theorems~\ref{thm:tw_circ} and~\ref{thm:td_circ}. 

Our first result shows that `circumference minus \(1\)' can be replaced with `cocircumference' in Theorem~\ref{thm:tw_circ}:

\begin{theorem}\label{TreewidthVsCocircumference}
 Every graph with at least one edge has treewidth at most its cocircumference. 
\end{theorem}

Our second result, which is the main result of this paper, is the following dual version of Theorem~\ref{thm:td_circ}.  

\begin{theorem}\label{PathwidthVsCocircumfrence}
  Let \(G\) be a \(2\)-connected graph with pathwidth $p$ and cocircumference $k$. 
  Then each edge of \(G\) is contained in a bond \(F\) such that \(p \le 3|F| - 2\).
  In particular, $G$ has pathwidth at most $3k-2$. 
\end{theorem}

Here also, $2$-connectivity is essential, as shown by complete ternary trees. 
Note that, contrary to Theorem~\ref{thm:td_circ}, there is no mention of treedepth in Theorem~\ref{PathwidthVsCocircumfrence}. 
This is because treedepth is not bounded from above by any function of the cocircumference, even when the graph is $2$-connected. 
(To see this, consider a cycle of length $n$: Its treedepth is roughly $\log_2 n$ but its cocircumference is $2$.) 

The two bounds in Theorems~\ref{thm:tw_circ} and~\ref{thm:td_circ} are best possible, as shown by complete graphs. 
By contrast, we do not know whether the bounds in Theorems~\ref{TreewidthVsCocircumference} and \ref{PathwidthVsCocircumfrence} are tight. 
However, they are within a constant factor of being optimal.  
For \cref{TreewidthVsCocircumference}, this follows from the existence of $n$-vertex cubic graphs with treewidth $\Omega(n)$ (random cubic graphs have this property w.h.p., or see e.g.~\cite{BEMPT04} for an explicit construction).  
For \cref{PathwidthVsCocircumfrence}, this follows from the following result. 

\begin{theorem}\label{LowerBound}
  For every integer \(k \ge 1\), there exists a $2$-connected graph \(G_k\)
  with pathwidth at least $k$ and cocircumference at most \(2k\).
\end{theorem}

The paper is organized as follows. 
First, we give the necessary definitions in Section~\ref{sec:prelim}. 
Then, we prove the three theorems above in Section~\ref{sec:proofs}. 
Finally, we end the paper in Section~\ref{sec:matroids} by putting our results in the broader context of matroids, and offer some general conjectures. 
Let us emphasize though that no matroid knowledge is necessary to read the proofs in Section~\ref{sec:proofs}.

\section{Preliminaries}
\label{sec:prelim}

In this paper, `graph' means a finite, simple graph. 
We refer the reader to the textbook by Diestel~\cite{Diestel5thEdition} for undefined graph theoretic terms and notations. 

A \emph{tree-decomposition} of a graph \(G\) is a family
\(\{X_u: u \in V(T)\}\) of subsets of \(V(G)\)
indexed by the vertex set of a tree \(T\) such that (i)
for every \(x \in V(G)\),
the set of all nodes \(u \in V(T)\) with \(x \in X_u\)
induces a (non-empty) subtree of \(T\), and (ii)
for every \(x y \in E(G)\) there exists \(u \in V(T)\)
with \(\{x, y\} \subseteq X_u\).
The \emph{width} of \(\{X_u: u \in V(T)\}\) is
\(\max \{|X_u| : u \in V(T)\} - 1\),
and the \emph{treewidth} of a graph \(G\), denoted
\(\tw(G)\), is the
minimum width of a tree-decomposition of \(G\).

A \emph{path-decomposition} of a graph \(G\) is
a tree-decomposition \(\{X_u: u \in V(P)\}\) of \(G\)
such that \(P\) is a path, and the \emph{pathwidth}
of a graph \(G\), denoted \(\pw(G)\),
is the minimum width of a
path-decomposition of \(G\).
We denote a path-decomposition \(\{X_u: u \in V(P)\}\)
with \(P = u_0 \cdots u_s\) by a sequence
\((X_{u_0}, \ldots, X_{u_s})\).

Recall that a bond in a graph $G$ is an inclusion-wise minimal set of edges $F$ such that $G-F$ has more  components than $G$. 
The following easy observation will be useful: A set $F$ of edges is a bond in \(G\) if and only if 
there exists a partition \(\{V_1, V_2\}\) of the vertex set of
a component of \(G\) such that each of the sets \(V_1\) and
\(V_2\) induces a connected subgraph of \(G\) and \(F\) is
the set of all edges of \(G\) between \(V_1\) and \(V_2\).
For a graph \(G\) with \(E(G) \neq \emptyset\), 
we denote by \(\cocircumference(G)\) its cocircumference, the size of a largest bond in \(G\). 

Finally, we note that multigraphs (that is, graphs with possibly parallel edges and loops) 
will be considered in the proof of \cref{LowerBound}. 
The definitions of path-decompositions and pathwidth extend to multigraphs in the natural way.

\section{The proofs}
\label{sec:proofs}

We begin with the proof of \cref{TreewidthVsCocircumference}. 

\begin{proof}[Proof of \cref{TreewidthVsCocircumference}]
The treewidth of a graph is the maximum treewidth of its components,
and the cocircumference of a graph
is the maximum cocircumference of
its components.
Thus, it is enough to prove the theorem in the case where the graph \(G\) is connected. 

Let \(T\) be a depth-first search tree of \(G\) rooted at some vertex
\(r\) of $G$.
Let \(X_r = \{r\}\), and for each \(u \in V(T) \setminus \{r\}\),
let \(X_u\) denote the set consisting of the parent of \(u\) in \(T\)
and all
descendants of \(u\) that are adjacent in \(G\) to
a proper ancestor of \(u\).
In particular, since \(u\) is adjacent in \(G\) to its parent in \(T\), we have \(u \in X_u\).
Let us verify that \(\{X_u : u \in V(T)\}\)
is a tree-decomposition of \(G\).

Since \(T\) is a depth-first search tree of \(G\), every edge of \(G\)
is between two distinct vertices on a root-to-leaf path in \(T\).
For each \(uv \in E(G)\) where \(v\) is a proper ancestor of \(u\),
if \(v'\) denotes the child of \(v\) that is an ancestor of \(u\),
then \(\{u, v\} \subseteq X_{v'}\).

For \(v \in V(G)\) and \(u \in V(T)\), we have \(v \in X_u\)
if and only if one of the following holds:
\begin{itemize}
  \item \(u = v\), or
  \item \(u\) is a child of \(v\), or
  \item \(u\) is a proper ancestor of \(v\) and there exists a
  proper ancestor of \(u\) that is adjacent to \(v\) in \(G\).
\end{itemize}
Hence, for every \(v \in V(G)\), the set
\(\{u \in V(T): v \in X_u\}\) induces a tree in \(T\),
so \(\{X_u : u \in V(T)\}\)
is indeed a tree-decomposition of \(G\).

For every vertex \(u \in V(T) \setminus \{r\}\) with a parent
\(u'\) in \(T\), removing the edge \(uu'\) splits \(T\) into
two trees.
Let \(F_u\) denote the set of all
edges of \(G\) between the vertex sets of these two trees.
The two components of \(T - uu'\)
are spanning trees of the components of \(G - F_u\),
so \(F_u\) is a bond.
Each vertex \(v \in X_u\) except the parent of \(u\)
is a descendant of
\(u\) adjacent to a proper ancestor
\(v'\) of \(u\) such that \(vv' \in F_u\), and
these edges \(vv'\) are distinct, so
\(|X_u| - 1 \le |F_u| \le \cocircumference(G)\).
Moreover, \(|X_r| - 1 = 0 \le \cocircumference(G)\), so
\[
\tw(G) \le \max \{|X_u| - 1 : u \in V(G)\} \le \cocircumference(G).\qedhere.
\]
\end{proof}

Next, we introduce some lemmas that will be used in the proof of \cref{PathwidthVsCocircumfrence}.   
A \emph{separation} of a graph \(G\) is a pair \((G_1, G_2)\)
such that \(G_1\) and \(G_2\) are edge-disjoint subgraphs of \(G\)
with \(G_1 \cup G_2 = G\).
A separation \((G_1, G_2)\) is \emph{trivial} if
\(E(G_1) = \emptyset\) or \(E(G_2) = \emptyset\).
The \emph{order} of a separation \((G_1, G_2)\) is
\(|V(G_1) \cap V(G_2)|\).
For a graph \(G\) and a vertex \(x\) of $G$, there exists a nontrivial
separation \((G_1, G_2)\) with \(V(G_1) \cap V(G_2) = \{x\}\)
if and only \(x\) is a cutvertex of \(G\).

\begin{lemma}\label{Almost2ConnectedFromSeparation}
  Let \(G\) be a connected graph without a cutvertex, and let
  \((G_1, G_2)\) be a nontrivial separation
  of \(G\) such that \(V(G_1) \cap V(G_2) = \{x, y\}\)
  for some vertices \(x\) and \(y\).
  Then none of the graphs \(G_1 + xy\) and \(G_2 + xy\) has
  a cutvertex.
\end{lemma}

\begin{proof}
  By symmetry, we only need to show that \(G_2 +xy\)
  does not have a cutvertex.
  Suppose to the contrary that \(G_2 + xy\) has a cutvertex.
  Let \((G_2', G_2'')\) be a nontrivial separation of
  \(G_2 + xy\) of order \(1\).
  Without loss of generality, \(xy \in E(G_2')\).
  Hence, \((G_1 \cup (G_2' - xy), G_2'')\) is a
  nontrivial separation of \(G\) of order \(1\),
  so \(G\) has a cutvertex, contradiction.
\end{proof}

For a graph \(G\) and \(x, y \in V(G)\),
let \(\pw(G; x)\) denote the minimum width
of a path-decomposition \((X_0, \ldots, X_s)\) of \(G\)
such that \(x \in X_0\), and let \(\pw(G; x, y)\)
denote the minimum width of a
path-decomposition \((X_0, \ldots, X_s)\) of \(G\)
such that \(x \in X_0\) and \(y \in X_s\).
Observe that \(\pw(G; x, y) \le \pw(G; x) + 1\),
and if \(x \neq y\), then \(\pw(G; x, y) \le \pw(G - y; x) + 1\).

\begin{lemma}\label{PathwidthOneToOne1Sum}
  Let \(G\) be a graph, let \(x \in V(G)\), and let
  \((G_1, G_2)\) be a separation of
  \(G\) such that \(x \in V(G_1)\) and
  \(V(G_1) \cap V(G_2) = \{x'\}\) for some
  \(x' \in V(G)\). Then
  \[
    \pw(G; x) \le \max\{\pw(G_1; x, x'), \pw(G_2; x')\}.
  \]
\end{lemma}

\begin{proof}
  Let \((X_0, \ldots, X_s)\) be a path-decomposition of \(G_1\)
  of width \(\pw(G_1; x, x')\) with \(x \in X_0\) and \(x' \in X_s\), and
  let \((Y_0, \ldots, Y_t)\) be a path-decomposition of \(G_2\)
  of width \(\pw(G_2; x')\) with \(x' \in Y_0\).
  Then the sequence
  \[
    (X_0, \ldots, X_s, Y_0, \ldots, Y_t)
  \]
  is a path-decomposition of \(G\)
  of width at most
  \(\max\{\pw(G_1; x, x'), \pw(G_2; x')\}\)
  and with \(x \in X_0\).
\end{proof}

\begin{lemma}\label{PathwidthOneToMany1Sum}
  Let \(G\) be a graph of the form
  \(B \cup G_1 \cup \cdots \cup G_m\), where \(m \ge 2\),
  the graphs \(G_1, \ldots, G_m\) are pairwise
  vertex-disjoint and each \(G_i\) shares exactly
  one vertex \(x_i\) with \(B\).
  Let \(p\) be an integer such that
  \(\pw(G_i; x_i) \le p\) for \(i \in \{2, \ldots, m\}\). 
  Then for every \(x \in V(B)\)
  we have
  \[
    \pw(G; x) \le \max\{\pw(B; x_1) + p + 1, \pw(G_1; x_1)\}.
  \]
\end{lemma}

\begin{proof}
  Let \((X_0, \ldots, X_s)\) be a path decomposition of \(B\)
  with \(x_1 \in X_0\) and width \(\pw(B; x_1)\), and
  for each \(i \in \{1, \ldots, m\}\), let
  \((Y_0^i, \ldots, Y_{t_i}^i)\) be a path decomposition of \(G_i\)
  with \(x_i \in Y_0^i\) and width $\pw(G_i; x_i)$. 
  For each \(v \in V(B)\), define
  \[
    j(v) = \min \{j \in \{0, \ldots, s\}: v \in X_{j}\}.
  \]
  Hence, \(j(x_1) = 0\).
  We may assume that \(j(v) \neq j(v')\) for distinct vertices
  \(v, v' \in V(B)\). 
  We construct an appropriate path-decomposition of \(G\)
  in three steps.
  First, we join the reversed path-decomposition of \(B\)
  with the path decomposition of \(G_1\) to obtain a path-decomposition
  of \(B \cup G_1\):
  \[
    (X_s, \ldots, X_0, Y_0^1, \ldots, Y_{t_1}^1).
  \]
  Next, we add \(x\) to the bags \(X_s, \ldots, X_0\):
  \[
    (X_s \cup \{x\}, \ldots, X_0 \cup \{x\}, Y_0^1, \ldots, Y_{t_1}^1).
  \]
  This is a path-decomposition of \(B \cup G_1\) with width
  \(\max\{\pw(B; x_1) + 1, \pw(G_1; x_1)\}\) and with $x$ appearing in the first bag. 
  We transform it into a path-decomposition of \(G\) by
  stacking it with the path-decompositions of the graphs
  \(G_2, \ldots, G_p\).
  More precisely, for each \(i \in \{2, \ldots, m\}\),
  let \(X_{j(x_i)}' = (X_{j(x_i)}\setminus \{x_i\}) \cup \{x\}\),
  and insert the new bags
  \(X_{j(x_i)}' \cup Y_0^i, \ldots, X_{j(x_i)}' \cup Y_{t_i}^i\)
  directly after the bag \(X_{j(x_i)} \cup \{x\}\). 
  By definition of \(j(x_i)\), this is a path-decomposition of \(G\),
  and the vertex \(x\) belongs to its first bag \(X_s \cup \{x\}\).
  It remains to estimate the width.
  We started with a path-decomposition of \(B \cup G_1\) of width
  at most \(\max\{\pw(B; x_1) + 1, \pw(G_1; x_1)\}\).
  For \(i \in \{2, \ldots, m\}\) and \(j \in \{0, \ldots, t_i\}\),
  we can bound the size of the bag \(X_{j(x_i)}' \cup Y_j^i\) as follows:
  \begin{align*}
    |X_{j(x_i)}' \cup Y_j^i|
      &\le |X_{j(x_i)}'| + |Y_j^i|\\
      &= |(X_{j(x_i)}\setminus \{x_i\}) \cup \{x\}| + |Y_j^i|\\
      &= |X_{j(x_i)}| + |Y_j^i|\\ 
      &\le (\pw(B; x_1) + 1) + (p + 1).
  \end{align*}
  The width of our path-decomposition of \(G\) is one less than
  the size of a largest bag, and thus is at most
  \(\max\{\pw(B; x_1) + p + 1, \pw(G_1; x_1)\}\).
  This proves the lemma.
\end{proof}

For two distinct vertices \(x\) and \(y\) of a graph \(G\),
an \emph{\(x\)--\(y\) bond} in \(G\) is a bond \(F\) in \(G\)
such that \(x\) and \(y\) are in distinct components of \(G - F\). 
The following lemma is the heart of the proof of \cref{PathwidthVsCocircumfrence}. 

\begin{lemma}\label{PathwidthVsCocircumfrenceHelper}
  Let \(G\) be a connected graph such that \(E(G) \neq \emptyset\),
  and let \(x, y \in V(G)\) be distinct vertices such that
  \(G + xy\) does not have a cutvertex.
  Then there exists an \(x\)--\(y\) bond \(F\) in \(G\)
  such that \(\pw(G - y; x) \le 3|F| - 2\),
  and if \(G\) does not have a cutvertex, then
  \(\pw(G - y; x) \le 3|F| - 3\).
\end{lemma}
\begin{proof}
  We prove the lemma by induction on \(|E(G)|\).
  In the base case, we have \(|E(G)| = 1\).
  Since \(G + xy\) does not have a cutvertex,
  we have \(V(G) = \{x, y\}\) and \(E(G) = \{xy\}\).
  Therefore, for the \(x\)--\(y\) bond \(F = \{xy\}\),
  we have \(\pw(G - y; x) = 0 = 3|F| - 3\).

  For the induction step, we assume that \(|E(G)| \ge 2\).
  Suppose first that there exists a nontrivial separation
  \((G_1, G_2)\) of \(G\) with \(V(G_1) \cap V(G_2) = \{x, y\}\).
  Unless \(xy \in E(G)\) and thus \(G + xy = G\),
  the separation \((G_1, G_2)\) of \(G\) can be transformed into a
  separation of \(G + xy\) by adding the edge
  \(xy\) to either of the sides.
  Hence, by \cref{Almost2ConnectedFromSeparation},
  none of the graphs \(G_1 + xy\) and \(G_2 + xy\) has a cutvertex.
  For each \(\alpha \in \{1, 2\}\),
  apply the induction hypothesis to \(G_\alpha\) to obtain
  an \(x\)--\(y\) bond \(F_\alpha\) in \(G_\alpha\)
  such that \(\pw(G_\alpha - y; x) \le 3|F_\alpha| - 2\).
  Then, \(F = F_1 \cup F_2\) is an \(x\)--\(y\) bond in \(G\)
  that satisfies \(|F| = |F_1| + |F_2| \ge \max\{|F_1|, |F_2|\} + 1\).
  By \cref{PathwidthOneToOne1Sum} applied to the separation
  \((G_1 - y, G_2 - y)\) of \(G - y\), we have
  \begin{align*}
    \pw(G - y; x)
      &\le \max\{\pw(G_1-y; x, x), \pw(G_2-y; x)\}\\
      &\le \max\{\pw(G_1-y; x) + 1, \pw(G_2-y; x)\}\\
      &\le \max\{3|F_1| - 1, 3|F_2| - 2\}\\
      &< 3\max\{|F_1|, |F_2|\}\\
      &= 3(\max\{|F_1|, |F_2|\}+1) - 3\\
      &\le 3|F|-3.  
  \end{align*}

  It remains to consider the case when \(G\) does not have a
  nontrivial separation \((G_1, G_2)\) with
  \(V(G_1) \cap V(G_2) = \{x, y\}\).
  Therefore, \(xy \not \in E(G)\), as otherwise
  we would have a nontrivial separation \((G[\{x, y\}], G - xy)\).
  Moreover, \(x\) is not a cutvertex of \(G - y\), because otherwise 
  there would exist a nontrivial
  separation \((G', G'')\) of \(G - y\) with
  \(V(G') \cap V(G'') = \{x\}\),
  so the separation
  \((G[V(G') \cup \{y\}], G[V(G'') \cup \{y\}])\)
  of \(G\) would contradict our assumption.
  Hence, \(x\) belongs to only one block of \(G - y\),
  which we denote by \(B\).

  Since \(G + xy\) does not have a cutvertex, 
  every component of \((G+xy) - (V(B) \cup \{y\})\)
  is adjacent in \(G\)
  to at least two vertices from \(V(B) \cup \{y\}\).
  As \(B\) is a block of \(G - y\),
  every component of \((G - y) - V(B)\)
  is adjacent in \(G-y\) to at most one vertex of \(B\).
  Since
  \[
    (G+xy) - (V(B) \cup \{y\}) = (G - y) - V(B) = G - (V(B) \cup \{y\}),
  \]
  this means that every component of
  \(G - (V(B) \cup \{y\})\) is adjacent 
  to \(y\) and to exactly one vertex of \(B\).

  Let \(\{x_1, \ldots, x_m\}\) denote
  the set of all vertices of \(B\) that
  have a neighbour outside \(B\) in \(G\).
  Note that if \(m = 1\), then
  \(x_1\) is a cutvertex of \(G\).
  For each \(i \in \{1, \ldots, m\}\),
  let \(H_i\) denote the subgraph of \(G\)
  induced by \(x_i\), \(y\), and the union
  of all components of \(G - (V(B) \cup \{y\})\)
  that are adjacent to \(x_i\).  
  (Note that there might be no such component, in which case \(H_i=G[\{x_i, y\}]\).)
  For each \(i \in \{1, \ldots, m\}\), let
  \(H_{-i} = \bigcup_{i' \neq i} H_{i'}\).
  By \cref{Almost2ConnectedFromSeparation} applied to
  \((B \cup H_{-i} + xy, H_i)\), the graph
  \(H_i + x_i y\) does not have a cutvertex,
  so by the induction hypothesis there exists
  an \(x_i\)--\(y\) bond \(F_i\) in \(H_i\)
  such that \(\pw(H_i - y; x_i) \le 3|F_i| - 2\).
  Observe that  \(\bigcup_{i=1}^m F_i\) is
  an \(x\)--\(y\) bond in \(G\).
  Without loss of generality, we assume that
  \(|F_1| \ge \cdots \ge |F_m|\).

  Suppose first that \(m = 1\). 
  Observe that \(x \neq x_1\) in this case. 
  Since \(B\) is a block of \(G - y\),
  it does not have
  a cutvertex, and neither does \(B + x x_1\).
  By the induction hypothesis,
  there exists an \(x\)--\(x_1\) bond \(F_0\) in \(B\)
  such that \(\pw(B - x_1; x) \le 3|F_0| - 3\).
  Both \(F_0\) and \(F_1\) are \(x\)--\(y\) bonds in \(G\);
  let \(F\) be a largest one among them.
  By \cref{PathwidthOneToOne1Sum} applied to the
  separation \((B, H_1 - y)\) of \(G - y\), we have
  \begin{align*}
    \pw(G - y; x)
      &\le \max \{\pw(B; x, x_1), \pw(H_1 - y; x_1)\}\\
      &\le \max \{\pw(B - x_1; x) + 1, \pw(H_1 - y; x_1)\}\\
      &\le \max \{3|F_0| - 2, 3|F_1| - 2\}\\
      &=3\max \{|F_0|, |F_1|\} - 2\\
      &=3|F| - 2, 
  \end{align*}
  which is enough since \(x_1\) is a cutvertex of \(G\).

  Now suppose that \(m \ge 2\).
  Again, since \(B\) is a block of \(G-y\),
  it does not have a cutvertex, and neither
  does \(B + x_1 x_2\).
  By the induction hypothesis,
  there exists an \(x_1\)--\(x_2\) bond \(F_0\) in \(B\)
  such that \(\pw(B - x_2; x_1) \le 3|F_0| - 3\).
  Let \(I\) denote the set of all
  \(i \in \{1, \ldots, m\}\) such that
  \(x_i\) belongs to the same component of
  \(B - F_0\) as \(x\).
  Observe that \(F_0 \cup \bigcup_{i \in I} F_i\)
  is an \(x\)--\(y\) bond in \(G\),
  and since the vertices \(x_1\) and \(x_2\)
  belong to distinct components of
  \(B - F_0\), we have \(I \cap \{1, 2\} \neq \emptyset\),
  and therefore
  \[
    \left|F_0 \cup \bigcup_{i \in I} F_i\right| \ge |F_0| + \min\{|F_1|,|F_2|\}=|F_0|+|F_2|.
  \]
  Furthermore, \(\bigcup_{i=1}^m F_i\) is an \(x\)--\(y\) bond
  in \(G\) of size at least \(|F_1| + 1\).
  Let \(F\) be a largest of the two \(x\)--\(y\) bonds
  \(F_0 \cup \bigcup_{i \in I} F_i\) and
  \(\bigcup_{i=1}^m F_i\).
  Therefore,
  \[
    |F| \ge \max \{|F_0|+|F_2|,|F_1|+1\}.
  \]
  For each \(i \in \{2, \ldots, m\}\), we have
  \(\pw(H_i - y; x_i) \le 3|F_i| - 2 \le 3|F_2| - 2\).
  By \cref{PathwidthOneToMany1Sum} applied to
  \(G - y = B \cup (H_1 - y) \cup \cdots \cup (H_m - y)\) and
  \(p = 3|F_2| - 2\), we have
  \begin{align*}
    \pw(G - y; x)
      &\le \max \{\pw(B; x_1) + p + 1, \pw(H_1 - y; x_1)\}\\
      &\le \max \{(\pw(B - x_2; x_1)+1) + p + 1, \pw(H_1 - y; x_1)\}\\
      &\le \max \{(3|F_0| - 2) + (3|F_2| - 2) + 1, 3|F_1| - 2\}\\
      &= \max \{3(|F_0| + |F_2|) - 3, 3(|F_1|+1) - 5\}\\
      &\le 3\max \{|F_0| + |F_2|, |F_1| + 1\} - 3\\
      &\le 3|F| - 3. \qedhere 
  \end{align*}
\end{proof}

We may now prove \cref{PathwidthVsCocircumfrence}. 

\begin{proof}[Proof of \cref{PathwidthVsCocircumfrence}]
  For every \(2\)-connected graph \(G\) with an edge
  \(xy\), we have \(G + xy = G\), so \(G + xy\) is
  \(2\)-connected. We have
  \(\pw(G) \le \pw(G -y; x) + 1\), and by
  \cref{PathwidthVsCocircumfrenceHelper},
  there exists an \(x\)--\(y\) bond \(F\) in \(G\) such 
  that \(\pw(G - y; x) \le 3|F| - 3\).
  Hence, \(\pw(G) \le 3|F| - 2\).
  Since \(F\) is an \(x\)--\(y\) bond, it must contain
  the edge \(xy\).
\end{proof}

We end this section with the proof of \cref{LowerBound}. 
For every integer \(k \ge 1\), we inductively define a
\(2\)-connected outerplane graph \(G_k\) with a distinguished
\emph{root edge} \(e_k\) incident with the outer face, as follows.
For \(k = 1\), the graph \(G_k\) is a cycle of length \(4\) and
we choose any of its edges as the root edge.
For \(k \ge 2\), the graph \(G_k\) is obtained from a cycle
with four edges \(e^1, e^2, e^3, e^4\) by gluing
to each of the edges \(e^1\), \(e^2\) and \(e^3\) a copy of \(G_{k-1}\)
along its root edge, and the root edge \(e_k\) of \(G_k\) is \(e^4\).
See Figure~\ref{fig:G3}.
\begin{figure}[ht!]
\centering
\includegraphics[width=0.5\textwidth]{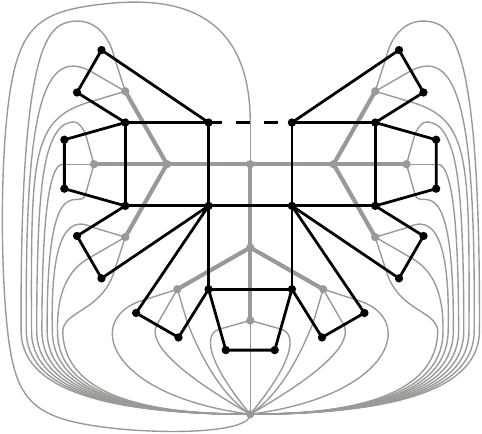}
\caption{The graph \(G_3\) and its plane dual \(G_3^*\).
The root edge \(e_3\) is dashed.}
\label{fig:G3}

\end{figure}

Let \(G_k^*\) denote the plane dual of \(G_k\),
and let \(a_k\) be the vertex of \(G_k^*\) corresponding to the outer face of
\(G_k\).
Observe that \(T_k := G_k^* - a_k\) is a complete ternary tree of height \(k-1\)  rooted at a vertex \(r_k\) 
(here, the \emph{height} of a tree is the maximum number of edges on a
root-to-leaf path), and \(G_k^*\) is the multigraph obtained from \(T_k\)
by adding the vertex \(a_k\), an edge between \(a_k\) and \(r_k\),
and three parallel edges between \(a_k\) and each leaf of \(T_k\).
It is well known that the pathwidth of a complete ternary tree of height $h$ is $h$, thus \(\pw(T_k) = k - 1\).
Since \(T_k \subseteq G_k^*\), this implies \(\pw(G_k^*) \ge k - 1\).
The following lemma strengthens this lower bound.

\begin{lemma}\label{DualLowerBound}
  For every integer \(k \ge 1\), \(\pw(G_k^*) \ge k\).
\end{lemma}
\begin{proof}
  Let \(e_k^*\) denote the edge of \(G_k^*\) between \(a_k\) and \(r_k\)
  which corresponds to the root edge \(e_k\) of \(G_k\).
  The graph \(G_k^* - e_k^*\) consists of the tree \(T_k\) and the vertex \(a_k\)
  attached to each leaf of \(T_k\) with three parallel edges.
  We prove by induction on \(k\) that \(\pw(G_k^* - e_k^*) \ge k\).
  This will imply that \(\pw(G_k^*) \ge k\).
  For \(k = 1\), the graph \(G_k^* - e_k^*\) consists of
  the vertices \(a_k^*\), \(r_k\), and three parallel edges between them,
  so \(\pw(G_k^* - e_k^*) = 1 = k\).
  Now, for the inductive step, suppose that \(k \ge 2\),
  and let \((X_0, \ldots, X_s)\) be a path-decomposition of \(G_k^* - e_k^*\)
  of minimum width, and among all such path-decompositions choose one minimizing \(s\).
  This way, each of the bags \(X_0\) and \(X_s\) contains at least two vertices.
  In particular, there exist vertices \(x_0 \in X_0\) and \(x_s \in X_s\) that are both distinct from \(a_k\).
  The graph \(G_k^* - r_k\), has three blocks \(B_1,B_2,B_3\),
  each of them is isomorphic to \(G_{k-1}^* - e_{k-1}^*\), and each pair of them
  intersects only in \(a_k\).
  Hence, each of the vertices \(x_0\) and \(x_s\) belongs to exactly one
  of the blocks \(B_1,B_2,B_3\).
  Without loss of generality, we assume that \(V(B_2) \cap \{x_0, x_s\} = \emptyset\).

  The graph \(B_2\) is a copy of \(G_{k-1}^* - e_{k-1}^*\) with a path decomposition
  \((X_0 \cap V(B_2), \ldots, X_s \cap V(B_2))\).
  By the induction hypothesis, the width of this path-decomposition is at least \(k - 1\),
  so there exists \(i \in \{0, \ldots, s\}\) such that
  \(|X_i \cap V(B_2)| \ge k \).
  The graph \(G_k^* - V(B_2)\) is a connected subgraph of \(G_k^* - e_k^*\) which intersects
  the bags \(X_0\) and \(X_s\) in the vertices \(x_0\) and \(x_s\), respectively.
  Hence, it must intersect the bag \(X_i\) in some vertex \(x_i\), and therefore,
  \[
   |X_i| \ge |X_i \cap V(B_2)| + |\{x_i\}| \ge k + 1.
  \]
  Therefore, the width of \((X_0, \ldots, X_s)\) is at least \(k\), and consequently,
  \(\pw(G_k^* - e_k^*) \ge k\). This completes the inductive proof.
\end{proof}

\begin{proof}[Proof of \cref{LowerBound}]
Bodlaender and Fomin~\cite{bodlaender2002approximation}
showed that for every \(2\)-connected
outerplane graph \(G\) without loops nor multiple edges,
we have \(\pw(G) \ge \pw(G^*)\), where $G^*$ denotes the plane dual of $G$. 
Hence, by \cref{DualLowerBound}, we have
\[
  \pw(G_k) \ge \pw(G_k^*) \ge k.
\]
The cocircumference of \(G_k\) is the circumference of
its plane dual \(G_k^*\).
Since \(G_k^* - a_k\) is a tree,
every cycle in \(G_k^*\) contains \(a_k\).
Therefore, the circumference of \(G_k^*\) is at most
the length of a longest path in \(T_k\) plus two:
\[
  \cocircumference(G_k) = \circumference(G_k^*) \le 2(k-1) + 2 = 2k.\qedhere
\]
\end{proof}

\section{Matroids}
\label{sec:matroids}

Before moving to matroids, we need to introduce a graph parameter closely related to treewidth, called branchwidth, which is defined as follows. 
A {\em branch decomposition} of a graph $G$ consists of a subcubic tree (an unrooted tree where every internal node has degree at most $3$) and a bijection between the leaves of the tree and the edges of $G$. 
If we remove an edge from the tree, the resulting two subtrees partition the edges of $G$ into two sets $A$ and $B$, and we are interested in the number of vertices of $G$ that appear as endpoints of edges in both sets $A$ and $B$.  
The {\em width} of the decomposition is the maximum of this number of common vertices, over every edge of the tree. 
The {\em branchwidth} of $G$, denoted $\bw(G)$, is the minimum width of a branch decomposition of $G$. 
While treewidth and branchwidth are not always equal, they are closely tied to each other: 
Robertson and Seymour~\cite{RS91} proved that
\begin{equation}
\label{eq:bw_vs_tw}
\bw(G) - 1 \leq \tw(G) \leq \left\lfloor \frac32 \bw(G) \right \rfloor - 1
\end{equation}
for every graph $G$ with $\bw(G) \geq 2$.  

The interest of branchwidth compared to treewidth is that its definition is phrased in terms of edges (on the leaves of the subcubic tree), instead of vertices (in the bags of the tree decomposition).  
This allows for a smooth generalization to matroids, which is as follows. 
Given a matroid $M$ with rank function $r$, a {\em branch decomposition} of $M$ consists of a subcubic tree and a bijection between the leaves of the tree and the elements of $M$.  
If we remove an edge from the tree, the resulting two subtrees partition the elements of $M$ into two sets $A$ and $B$, and this time we are interested in the following measure: $r(A) + r(B) - r(M) +1$. 
The {\em width} of the decomposition is the maximum of this measure over every edge of the tree, and the {\em branchwidth} of $M$, denoted $\bw(M)$, is the minimum width of a branch decomposition of $M$. 

At this point, it might come as a surprise that the branchwidth of graph $G$ is not always equal to the branchwidth of its cycle matroid $M(G)$. 
However, this is essentially the case: Mazoit and Thomass\'e~\cite{MT07} proved that for every graph $G$, 
either 
\begin{equation}
\label{eq:bw}
\bw(G) = \bw(M(G)),
\end{equation}
or $\bw(G) = 2$ and $\bw(M(G))=1$. 
Furthermore, \eqref{eq:bw} holds if $G$ has at least one cycle. 

One key feature of branchwidth, which follows from its definition, is that the branchwidth of a matroid $M$ is equal to the branchwidth of its dual matroid $M^*$: 
\begin{equation}
\label{eq:bw_dual}
\bw(M) = \bw(M^*). 
\end{equation}

Lastly, we note that circumference and cocircumference for matroids are defined as expected: 
The {\em (co)circumference} of a matroid $M$ with at least one (co)circuit is the maximum size of a (co)circuit of $M$. 
Note that the cocircumference of $M$ is equal to the circumference of its dual matroid $M^*$.  
Note also that for a graph $G$, the (co)circumference of $G$ is equal to that of its cycle matroid $M(G)$. 

With these notions in hand, we may now state our first conjecture: 

\begin{conjecture}
\label{conj:bw}
Every matroid $M$ with circumference $k$ has branchwidth $O(k)$. 
\end{conjecture}

For graphic matroids, Theorem~\ref{thm:tw_circ} combined with \eqref{eq:bw_vs_tw} and \eqref{eq:bw} show that the conjecture is true in this case. 
For cographic matroids, the conjecture is true as well, as follows from \cref{TreewidthVsCocircumference} combined with \eqref{eq:bw_vs_tw}, \eqref{eq:bw}, and \eqref{eq:bw_dual}. 
For general matroids, the conjecture is wide open though.  

Next, we wish to discuss a potential generalization of \cref{thm:td_circ} and \cref{PathwidthVsCocircumfrence}.  
Following~\cite{GGW06}, the pathwidth of a matroid $M$ is defined as follows. 
Given an ordering $e_1, e_2, \dots, e_m$ of the elements of $M$, we consider 
for each $i\in \{1, \dots, m\}$ the measure $r(A) + r(B) - r(M)$ with $A=\{e_1, \dots, e_i\}$ and $B=\{e_{i+1}, \dots, e_m\}$, and take its maximum over $i\in \{1, \dots, m\}$. 
This is the {\em width} of the ordering. 
The {\em pathwidth} of $M$ is the minimum width of an ordering of its elements, and is denoted $\pw(M)$.  
Observe that, except for the $+1$ missing in the width measure above, this definition is almost the same as that of branchwidth with the extra requirement that the subcubic tree be a caterpillar (a path plus some hanging leaves). 
Using this connection, the following upper bound can be shown (see~\cite{kashyap2008matroid} for a proof): 
\begin{equation}
\label{eq:bw_vs_pw}
\bw(M) \leq \pw(M) + 1
\end{equation}
holds for every matroid $M$. 

Also, note that matroid pathwidth is self-dual (as follows from the definition): 
\begin{equation}
\label{eq:pw_dual}
\pw(M) = \pw(M^*) 
\end{equation}
holds for every matroid $M$. 

Finally, given a graph $G$, the pathwidth of its cycle matroid is at most the pathwidth of $G$ plus one.
In fact, Kashyap~\cite{kashyap2008matroid} constructed from any graph \(G\) a multigraph \(\overline{G}\) with \(G \subseteq \overline{G}\)
such that \(\pw(M(\overline{G})) = \pw(G)+1\), so
\begin{equation}
\label{eq:pw_G_vs_pw_M}
\pw(M(G)) \le \pw(G) + 1. 
\end{equation}

In \cref{thm:td_circ} and \cref{PathwidthVsCocircumfrence}, it was essential that the graph $G$ under consideration is $2$-connected. 
Thus, if we hope to bound the pathwidth of a matroid from above by a function of its circumference, we need to add a connectivity requirement. 
Recall that a matroid is {\em connected} if every two elements lie in a common circuit. 
Also, a matroid is connected if and only if its dual is. 
Since a graph is $2$-connected precisely when its cycle matroid is connected, this naturally leads us to the following conjecture.  

\begin{conjecture}
\label{conj:pw}
Every connected matroid $M$ with circumference $k$ has pathwidth $O(k)$. 
\end{conjecture}

Using \eqref{eq:bw_vs_pw}, it is not difficult to show that \cref{conj:pw} in fact implies \cref{conj:bw}, this second conjecture is thus stronger. 
\cref{conj:pw} is true for graphic matroids, as follows from \cref{thm:td_circ} and \eqref{eq:pw_G_vs_pw_M}. 
The conjecture is also true for cographic matroids, as follows from \cref{PathwidthVsCocircumfrence} combined with  \eqref{eq:pw_dual} and \eqref{eq:pw_G_vs_pw_M}.  
We are not aware of other cases where \cref{conj:pw} is known to be true. 

We note that, for general matroids, it is possible to show a $O(k^2)$ bound 
for \cref{conj:pw}, and thus also for \cref{conj:bw}, by using a theorem of Seymour (see~\cite{DOO95}), stating that contracting a largest circuit in a connected matroid decreases its circumference by at least $1$. 
Alternatively, a quadratic upper bound can also be derived from Corollary~3.16 in~\cite{KKLM17} (see also Theorem~9 in \cite{BrianskiKKPS22} and the discussion around it). 

\section*{Acknowledgments}

We thank the two anonymous referees for their helpful remarks on a previous version of the manuscript. 

\printbibliography
\end{document}